\documentclass[a4paper,10pt]{article}
\usepackage{amsmath,amsthm,xypic,
amssymb,amsfonts,
pdfsync,stmaryrd,mathrsfs,yfonts}
\usepackage[capitalise]{cleveref}
\newtheorem{theorem}{Theorem}[section]
\newtheorem{prop}[theorem]{Proposition}
\newtheorem{cor}[theorem]{Corollary}
\newtheorem{lemma}[theorem]{Lemma}
\theoremstyle{definition}
\newtheorem{definition}[theorem]{Definition}
\newtheorem{remark}[theorem]{Remark}

\newtheorem{exm}[theorem]{Example}
\newtheorem{question}[theorem]{Question}
\theoremstyle{plain}

\newcommand{\rad}{\mathop{\mathrm{rad}}}

\newcommand{\trd}{\mathop{\mathrm{Trd}}}

\renewcommand{\phi}{\varphi}

\newcommand{\symd}{\mathop{\mathrm{Symd}}}
\newcommand{\sym}{\mathop{\mathrm{Sym}}}

\newcommand{\car}{\mathop{\mathrm{char}}}
\newcommand{\id}{\mathop{\mathrm{id}}}

\newcommand{\an}{\mathop{\mathrm{an}}}
\newcommand{\met}{\mathop{\mathrm{met}}}
\author{A.-H. Nokhodkar}
\title{Quadratic $D$-forms with applications to hermitian forms}
\begin{document}
\maketitle

\begin{abstract}
We study some properties of quadratic forms with values in a field whose underlying vector spaces are endowed with the structure of right vector spaces over a division ring extension of that field.
Some generalized notions of isotropy, metabolicity and isometry are introduced and used to find a Witt decomposition for these forms.
We then associate to every (skew) hermitian form over a division algebra with involution of the first kind a quadratic form defined on its underlying vector space.
It is shown that this quadratic form, with its generalized notionsو of isotropy and isometry, can be used to determine the isotropy behaviour and the isometry class of (skew) hermitian forms.\\
\\
\noindent
\emph{Mathematics Subject Classification:} 11E39, 11E04, 11E81, 16W10. \\
\emph{Keywords:}  Quadratic form, hermitian form, Witt decomposition, right
vector space, division algebra.\\
\end{abstract}

\section{Introduction}
An important problem in the theory of hermitian forms is to associate some quadratic form over the base field to a hermitian form capturing information about it.
This is inspired by the fact that hermitian forms over division algebras with involution are natural generalizations of bilinear and quadratic forms.
Using some results in the literature, one can find a solution to this problem in certain special cases, either in its current form or in the context of central simple algebras with involution (see for example \cite{jacob}, \cite{sah},
 \cite[\S 16]{knus}, \cite{bayer}, \cite{siv}, \cite{tignol}, \cite{dol}, \cite{mei} and \cite{meq}).
Among them, the most elementary case is the Jacobson's construction \cite{jacob} which can be summarized by saying that the theory of hermitian forms over a quaternion division algebra (or a quadratic separable extension) with the canonical involution reduces to the theory of quadratic forms.

The Jacobson's construction was generalized in \cite{meh} by associating to every (skew) hermitian form a system of quadratic forms given by transfers via linear maps from the space of (skew) symmetric elements to the base field.
This construction determines the isometry class of (skew) hermitian forms and their isotropy behaviour.
Also, using this system and the analogue of Springer's theorem for a system of two quadratic forms, it was shown in \cite{meh} that an anisotropic hermitian form over a quaternion algebra with involution of the first kind in characteristic two remains anisotropic over all odd degree extensions of the ground field.

In this work, we study another generalization of the Jacobson's construction.
Let $F$ be a field and let $(D,\sigma)$ be a division algebra with involution of the first kind over $F$.
Our aim is to associate to every (skew) hermitian space $(V,h)$ over $(D,\sigma)$ a quadratic form on $V$ with values in $F$ which reflects some important properties of $h$.
Such a quadratic form has an additional property, namely that its underlying vector space is endowed with the structure of a right vector space over $D$.
The approach we follow is to enrich the structure of quadratic forms with this property such that the $D$-structure of $V$ is taken into account.
This point of view will be developed in section \ref{sec-dforms} by introducing a class of quadratic forms, called quadratic $D$-forms.
Other related notions for quadratic forms, such as $D$-isometry, $D$-isotropy and $D$-metabolicity will also be defined accordingly.
In section \ref{sec-witt}, we study a decomposition theorem for quadratic $D$-forms, which is a generalization of the classical Witt decomposition theorem (see \cref{main}).
These elementary results show that the study of quadratic $D$-forms is interesting in itself.
They also motivate one to generalize other important properties of quadratic forms to the context of quadratic $D$-forms.
For example, one may consider the question of whether a $D$-anisotropic quadratic $D$-form remains $D$-anisotropic over odd degree extensions of the base field  (see \cref{ques}).

In section \ref{sec-pi}, we use quadratic $D$-forms to classify hermitian and skew hermitian forms.
Let $(D,\sigma)$ be a division algebra with involution of the first kind over a field $F$ and
let $(V,h)$ be a $\lambda$-hermitian space over $(D,\sigma)$, where $\lambda=\pm1$.
Suppose that either $D\neq F$ or $\lambda\neq-1$.
For every $F$-linear map $\pi:\sym_\lambda(D,\sigma)\rightarrow F$ whose restriction to $\symd_\lambda(D,\sigma)$ is nontrivial, we associate a quadratic $D$-form $q_{h,\pi}$ to $h$.
It is shown in \cref{isom} that two $\lambda$-hermitian forms $h$ and $h'$ are isometric if and only if $q_{h,\pi}$ and $q_{h',\pi}$ are $D$-isometric.
It is also shown in \cref{iso} and \cref{met} that $h$ is isotropic (resp. metabolic) if and only if $q_{h,\pi}$ is $D$-isotropic (resp. $D$-metabolic).
Using this, one may consider \cref{ques} as a generalization of the question of whether an anisotropic $\lambda$-hermitian form over $(D,\sigma)$ remains anisotropic over all odd degree extensions of $F$.
Finally, we consider quadratic forms which admits $D$ in section \ref{sec-admit}, to show that some familiar quadratic forms studied earlier in the literature are examples of quadratic $D$-forms.
It is shown that in the case where $\pi$ satisfies certain symmetry property, a quadratic $D$-form can be realised as the $\pi$-invariant of a $\lambda$-hermitian form if and only if it admits $D$ (see \cref{real}).

\section{Preliminaries}
Let $V$ be a vector space of finite dimension over a field $F$.
A {\it quadratic form} on $V$ is a map $q:V\rightarrow F$ satisfying $q(a u+bv)=a^2q(u)+b^2q(v)+ab\mathfrak{b}_q(u,v)$ for every $u,v\in V$ and $a,b\in F$, where $\mathfrak{b}_q:V\times V\rightarrow F$ is a bilinear form.
The pair $(V,q)$ is called a {\it quadratic space} over $F$ and the bilinear form $\mathfrak{b}_q$ is called the {\it polar form} of $q$.
For a subspace $W$ of $V$ we use the notation
\[W^{\perp_q}=\{v\in V\mid \mathfrak{b}_q(v,w)=0\ {\rm for \ all} \ w\in W\}.\]
We will simply denote $W^{\perp_q}$ by $W^\perp$ if the form $q$ is clear from the context.
Set $\rad V=V^\perp$.
The form $q$ is called {\it nonsingular} if $\rad V=\{0\}$, or equivalently, $\mathfrak{b}_q$ is nondegenerate.
Also, $q$ is called {\it totally singular} if $\mathfrak{b}_q$ is trivial.
A subspace $W$ of $V$ is called {\it nonsingular} if $W\cap W^\perp=\{0\}$, or equivalently, $q|_W$ is nonsingular.

The orthogonal sum of two quadratic spaces $(V,q)$ and $(V',q')$ is denoted by $(V\perp V',q\perp q')$.
An {\it  isometry} between $(V,q)$ and $(V',q')$, denoted by $(V,q)\simeq(V',q')$ or $q\simeq q'$ for short, is an isomorphism $f:V\rightarrow V'$ of vector spaces satisfying $q'(f(v))=q(v)$ for all $v\in V$.
For $\alpha \in F^\times$, the {\it scaled} quadratic space $(V,\alpha\cdot q)$ is defined by $\alpha\cdot q(v)=\alpha q(v)$ for all $v\in V$.
We will simply denote $(-1)\cdot q$ by $-q$.
Also, for $a_1,\cdots,a_n\in F^\times$ the quadratic form $a_1x_1^2+\cdots+a_nx_n^2$ is denoted by $\langle a_1,\cdots,a_n\rangle$.

A quadratic space $(V,q)$ (or the form $q$ itself) is called {\it isotropic} if there exists a nonzero vector $v\in V$ such that $q(v)=0$.
Otherwise, it is called {\it anisotropic}.
A subspace $W$ of $V$ is called {\it totally isotropic} if $q|_W$ is trivial.
A nonsingular quadratic space $(V,q)$ (or the form $q$ itself) is called {\it metabolic} if there exists a totally isotropic subspace $L$ of $V$ such that $\dim_DL=\frac{1}{2}\dim_DV$.
Such a subspace $L$ is called a {\it lagrangian} of $(V,q)$.
Note that $L^\perp=L$ for every lagrangian $L$ of $(V,q)$.

Let $A$ be a central simple algebra over a field $F$.
An {\it involution} on $A$ is an antiautomorphism $\sigma$ of $A$ satisfying $\sigma^2=\id$.
An involution $\sigma$ on $A$ is said to  be of {\it the first kind} if $\sigma|_F=\id$, and of {\it the second kind} otherwise.
An involution $\sigma$ of the first kind is called {\it symplectic} if it becomes adjoint to an alternating bilinear form over a splitting field of $A$.
Otherwise, it is called {\it orthogonal}.
For an algebra with involution $(A,\sigma)$ and $\lambda=\pm1$ we use the notation
\begin{align*}
  {\sym}_\lambda(A,\sigma)&=\{x\in A\mid \sigma(x)=\lambda x\},\\
  {\symd}_\lambda(A,\sigma)&=\{x+\lambda\sigma(x)\mid x\in A\}.
\end{align*}
It is easy to see that if $s\in\symd_\lambda(A,\sigma)$ then $\sigma(d) sd\in\symd_\lambda(A,\sigma)$ for all $d\in D$.
We will drop the index $\lambda$ from $\sym_\lambda(A,\sigma)$ and $\symd_\lambda(A,\sigma)$ if $\lambda=1$.

Let $(D,\sigma)$ be a finite dimensional division algebra with involution of the first kind over a field $F$ and let $\lambda=\pm1$.
Let $V$ be a finite dimensional right vector space over $D$.
A {\it $\lambda$-hermitian form} on $V$ is a bi-additive map $h:V\times V\rightarrow D$ such that $h(u\alpha,v\beta)=\sigma(\alpha) h(u,v)\beta$ and $h(v,u)=\lambda\sigma(h(u,v))$ for all $u,v\in V$ and $\alpha,\beta\in D$.
The pair $(V,h)$ is called a {\it $\lambda$-hermitian space} over $(D,\sigma)$.
If $\lambda=1$ (resp. $\lambda=-1$), we will call $h$ a {\it hermitian} (resp. {\it skew hermitian}) {\it form}.
Note that if $(D,\sigma)=(F,\id)$ then a $\lambda$-hermitian form on $(D,\sigma)$ is a symmetric or antisymmetric bilinear form over $F$.

A $\lambda$-hermitian space $(V,h)$ (or the form $h$ itself) is called {\it nondegenerate} if there is no nonzero vector $u\in V$ such that $h(u,v)=0$ for all $v\in V$.
The form $h$ is called {\it diagonalizable} if $(V,h)$ has an {\it orthogonal basis}, i.e.,
a basis $(v_1,\cdots,v_n)$ of $V$ over $D$ satisfying $h(v_i,v_j)=0$ for all $i\neq j$.
In this case, the form $h$ is  denoted by $\langle\alpha_1,\cdots,\alpha_n\rangle_{(D,\sigma)}$, where $\alpha_i=h(v_i,v_i)\in D$ for $i=1,\cdots,n$.

Let $(V,h)$ be a $\lambda$-hermitian space over $(D,\sigma)$.
It is easily seen that $h(v,v)\in\sym_\lambda(D,\sigma)$ for all $v\in V$.
We say that $(V,h)$ (or the form $h$ itself) is {\it even} if $h(v,v)\in\symd_\lambda(D,\sigma)$ for every $v\in V$.
If $\car F\neq2$ then all $\lambda$-hermitian forms are even, because $\sym_\lambda(D,\sigma)=\symd_\lambda(D,\sigma)$ in this case.

A $\lambda$-hermitian space $(V,h)$ (or the form $h$) is called {\it isotropic} if there exists a nonzero vector $v\in V$ such that $h(v,v)=0$.
Such a vector $v$ is called an {\it isotropic} vector.
The form $h$ is called {\it anisotropic} if it is not isotropic.
A nondegenerate $\lambda$-hermitian space $(V,h)$ (or the form $h$) is called {\it metabolic} if there exists a subspace $L$ of $V$ with $\dim_DL=\frac{1}{2}\dim_DV$ such that $h|_{L\times L}$ is trivial.
Such a subspace $L$ is called a {\it lagrangian} of $(V,h)$.

The orthogonal sum of two $\lambda$-hermitian spaces $(V,h)$ and $(V',h')$ over $(D,\sigma)$ is denoted by $(V\perp V',h\perp h')$.
An {\it  isometry} between $(V,h)$ and $(V',h')$, denoted by $(V,h)\simeq(V',h')$ or simply $h\simeq h'$, is an isomorphism $f:V\rightarrow V'$ of right vector spaces  over $D$ satisfying $h'(f(u),f(v))=h(u,v)$ for all $u,v\in V$.

\section{Quadratic $D$-forms}\label{sec-dforms}
Throughout this section, $F$ denotes a field of arbitrary characteristic and $D$ denotes a finite dimensional division algebra over $F$.

Let $V$ be a finite dimensional right vector space over $D$.
Then $V$ is also a vector space over $F$ and we may consider a quadratic form $q:V\rightarrow F$.
We say that $q$ is a {\it quadratic $D$-form} if $W^\perp$ is a vector space over $D$ for every $D$-subspace $W$ of $V$.
In this case, we say that $(V,q)$ is a {\it quadratic $D$-space}.
Note that if $D=F$, then a quadratic $D$-form is just a quadratic form.
Also, if $(V,q)$ is a quadratic $D$-space and $W$ is a $D$-subspace of $V$, then $q|_W$ is a quadratic $D$-form.
This follows from the equality $S^{\perp_{q|_W}}=S^{\perp_q}\cap W$ for every subspace $S$ of $W$.

\begin{lemma}\label{lemnew}
Let $V$ be a one-dimensional right vector space over $D$.
A quadratic form $q:V\rightarrow F$ is a quadratic $D$-form if and only if it is either nonsingular or totally singular.
\end{lemma}

\begin{proof}
The claim follows from the fact that $V$ has exactly two $D$-subspaces.
\end{proof}

\begin{cor}\label{vd}
Let $(V,q)$ be a quadratic $D$-space.
Then for every $v\in V$, $q|_{vD}$ is either nonsingular or totally singular.
\end{cor}

\begin{proof}
The result follows from \cref{lemnew}.
\end{proof}

\begin{remark}\label{rems}
The orthogonal sum of quadratic $D$-spaces is not necessarily a quadratic $D$-space if $D\neq F$.
To construct a counter example, let $\{d_1,\cdots,d_n\}$ be a basis of $D$ over $F$, where $n=\dim_FD\geqslant4$.
Consider the $F$-subspace
\[U=d_3F+\cdots+d_nF\subseteq D.\]
Let $\rho$ be a nonsingular quadratic form on $U$.
Let $\phi$ and $\phi'$ be nonsingular quadratic forms on $d_1F+d_2F$ satisfying
\begin{equation}\label{eq10}
\mathfrak{b}_\phi(d_1,d_2)=1\quad {\rm and} \quad \mathfrak{b}_{\phi'}(d_1,d_2)\neq1.
\end{equation}
Set $q=\phi\perp\rho$ and $q'=\phi'\perp\rho$.
Then $q$ and $q'$ are quadratic $D$-forms on $V$ by \cref{lemnew}.
However, $q\perp (-q')$ is not a quadratic $D$-form.
Indeed, in the quadratic space $(V\perp V,q\perp (-q'))$, one has \[(d_3,d_3)\in ((1,1)D)^\perp,\]
while $(d_1,d_1)\notin ((1,1)D)^\perp$, because
$\mathfrak{b}_{q\perp (-q')}((d_1,d_1),(d_2,d_2))\neq0$
by (\ref{eq10}).
\end{remark}

\begin{definition}
Let $(V,q)$ and $(V',q')$ be two quadratic $D$-spaces.
We say that $q$ and $q'$ are {\it $D$-compatible} if $q\perp q'$ is a quadratic $D$-form.
\end{definition}

Let $V$ and $V'$ be two finite-dimensional right vector spaces over $D$ and let $q:V\rightarrow F$ and $q':V'\rightarrow F$ be quadratic forms.
We say that $q$ is {\it $D$-isometric} to $q'$ if there exists an isomorphism $f:V\rightarrow V'$ of right vector spaces over $D$ such that $q'(f(v))=q(v)$ for every $v\in V$.
In this case, we write $(V,q)\simeq_D(V',q')$, or simply $q\simeq_D q'$.
Also, the map $f$ is called a {\it $D$-isometry}.
It is readily verified that if $q\simeq_D q'$ then $q$ is a quadratic $D$-form if and only if $q'$ is a quadratic $D$-form.

\begin{lemma}\label{orth}
Let $(V,q)$ be a quadratic $D$-space and let $W\subseteq V$ be a nonsingular $D$-subspace of $V$.
Then $(V,q)\simeq_D (W,q|_W)\perp (W^\perp,q|_{W^\perp})$.
In particular, $q|_W$ and $q|_{W^\perp}$ are $D$-compatible.
\end{lemma}

\begin{proof}
Since $W$ is nonsingular, we have $W+W^\perp=V$ and $W\cap W^\perp=\{0\}$.
As $W$ and $W^\perp$ are $D$-subspaces of $V$, the natural isometry $f:(W,q|_W)\perp (W^\perp,q|_{W^\perp})\simeq (V,q)$ defined by $f((w,w'))=w+w'$ is an isomorphism of right vector spaces.
Hence, it is a $D$-isometry.
\end{proof}

The following result is similarly verified.
\begin{lemma}\label{dec}
Let $(V,q)$ be a quadratic $D$-space and let $W$ be any $D$-subspace of $V$ for which $V=\rad V\oplus W$.
Then $q|_W$ is nonsingular and $q\simeq_D q|_{\rad V}\perp q|_W$.
\end{lemma}

Let $V$ be a right vector space over $D$ and let $q:V\rightarrow F$ be a quadratic form.
A basis $\{v_1,\cdots,v_n\}$ of $V$ over $D$ is called an {\it orthogonal $D$-basis} of $(V,q)$ if for every $i\neq j$, the $D$-subspace $v_iD$ is orthogonal to $v_jD$, i.e., $\mathfrak{b}_q(v_id,v_jd')=0$ for every $d,d'\in D$.
In this case, we say that $(V,q)$ is {\it $D$-diagonalizable}.

\begin{lemma}\label{total}
Let $(V,q)$ be a quadratic $D$-space.
Suppose that either $D\neq F$ or $\car F\neq2$.
If the restriction $q|_{vD}$ is totally singular for all $v\in V$, then $q$ is totally singular.
\end{lemma}

\begin{proof}
Observe first that if $\car F\neq2$, the claim is evident, because $q$ is trivial in this case.
Suppose that $\car F=2$, and hence $D\neq F$.
The hypothesis implies that
$\mathfrak{b}_q(v,vd)=0$ for all $v\in V$ and $d\in D$.
Applying this relation to the vector $u+v\in V$, one concludes that
\begin{equation}\label{eq11}
\mathfrak{b}_q(u,vd)=\mathfrak{b}_q(ud,v)\quad {\rm for\ all} \ u,v\in V \ {\rm and}\ d\in D.
\end{equation}
Let $d_1,d_2\in D$ and $u,v\in V$.
Then using (\ref{eq11}) we have
\begin{align*}
\mathfrak{b}_q(u,vd_1d_2)&=\mathfrak{b}_q(u,(vd_1)d_2)=\mathfrak{b}_q(ud_2,vd_1)\\
&=\mathfrak{b}_q((ud_2)d_1,v)=\mathfrak{b}_q(u(d_2d_1),v)=\mathfrak{b}_q(u,vd_2d_1).
\end{align*}
Hence,
\begin{equation}\label{eq12}
\mathfrak{b}_q(u,v(d_1d_2-d_2d_1))=0 \quad {\rm for\ all} \ u,v\in V \ {\rm and}\ d_1,d_2\in D.
\end{equation}
Choose $d_1,d_2\in D$ such that $d_1d_2\neq d_2d_1$ and set $d'=d_1d_2-d_2d_1\in D$.
Let $v\in V$ be an arbitrary vector.
Then (\ref{eq12}) implies that
$\mathfrak{b}_q(u,vd')=0$ for every $u\in V$, hence $vd'\in \rad V$.
Since $\rad V$ is a $D$-space and $d'\neq0$, one concludes that $v\in \rad V$.
It follows that $\rad V=V$, i.e., $q$ is totally singular.
\end{proof}

\begin{prop}
If $D\neq F$ or $\car F\neq2$ then every quadratic $D$-space is $D$-diagonalizable.
\end{prop}
\begin{proof}
Let $(V,q)$ be a quadratic $D$-space.
In view of \cref{dec}, it suffices to consider the case where $q$ is nonsingular.
By \cref{total} and \cref{vd}, there exists $v\in V$ such that $q|_{vD}$ is nonsingular.
Hence, $q\simeq_D q|_{vD}\perp q|_{(vD)^\perp}$ by \cref{orth}.
The result now follows by induction on $\dim_DV$.
\end{proof}

\section{Witt decomposition of quadratic $D$-forms}\label{sec-witt}
We continue to assume that $D$ is a finite dimensional division algebra over a field $F$.

Let $V$ be a right vector space over $D$ and let $q:V\rightarrow F$ be a quadratic form.
A nonzero vector $v\in V$ is called {\it $D$-isotropic} if $q|_{vD}=0$, i.e., $vD$ is a totally isotropic subspace of $V$.
We say that $(V,q)$ (or simply $q$) is {\it $D$-isotropic} if there exists a $D$-isotropic vector $v\in V$.
Otherwise, $q$ is called {\it $D$-anisotropic}.
We say that $(V,q)$ (or the form $q$ itself) is {\it $D$-metabolic} if
$(i)$ $q$ is nonsingular;
$(ii)$ there exists a totally isotropic $D$-subspace $L$ of $V$ such that $\dim_DL=\frac{1}{2}\dim_DV$.
Such a subspace $L$ is called a {\it $D$-lagrangian} of $(V,q)$.
Note that for every $D$-lagrangian $L$ of $(V,q)$ we have $L^\perp=L$.
Clearly, the $D$-isotropy and $D$-metabolicity are preserved under $D$-isometry.

\begin{lemma}\label{lemma}
Let $(V,q)$ be a nonsingular quadratic $D$-space.
Let $v\in V$ be a $D$-isotropic vector.
Then for every $w\in V\setminus (vD)^\perp$, the restriction $q|_{vD+wD}$ is $D$-metabolic.
\end{lemma}

\begin{proof}
Set $W=vD+wD$.
Since $vD\subseteq (vD)^\perp$ and $w\notin (vD)^\perp$, $W$ is a two-dimensional $D$-subspace of $V$.
Hence, it suffices to show that $q|_{W}$ is nonsingular.
Let $u=vd_1+wd_2\in W\cap W^\perp$, where $d_1,d_2\in D$.
The relation $\mathfrak{b}_q(u,vd)=0$ implies that $\mathfrak{b}_q(wd_2,vd)=0$ for all $d\in D$, because $q|_{vD}=0$.
Since $w\notin (vD)^\perp$ and $(vD)^\perp$ is a $D$-subspace, we obtain $d_2=0$, hence $u=vd_1$.
As $w\in W$, one has $vd_1\in (wD)^\perp$.
If $d_1\neq0$, then $vD\subseteq(wD)^\perp$, because $(wD)^\perp$ is a $D$-subspace of $V$.
This contradicts $w\notin (vD)^\perp$.
Hence, $d_1=0$, which implies that $W\cap W^\perp=\{0\}$.
\end{proof}

\begin{cor}\label{dis}
Let $(V,q)$ be a nonsingular quadratic $D$-space.
If $q$ is $D$-isotropic then there exists a $D$-subspace $W$ of $V$ with $\dim_DW=2$ such that $q|_W$ is $D$-metabolic and $q\simeq_Dq|_W\perp q|_{W^\perp}$.
Moreover, if $v$ is a $D$-isotropic vector of $(V,q)$ then the subspace $W$ can be chosen in such a way that $v\in W$.
\end{cor}

\begin{proof}
Choose a vector $w\in V$ such that $\mathfrak{b}_q(v,w)\neq0$.
Then the $D$-subspace $W=vD+wD$ is the required subspace, thanks to Lemmas \ref{lemma} and \ref{orth}.
\end{proof}

\begin{lemma}\label{2}
Every $D$-metabolic quadratic $D$-form is $D$-isometric to ortho\-gonal sums of two-dimensional $D$-metabolic quadratic $D$-forms.
\end{lemma}

\begin{proof}
Let $(V,q)$ be a $D$-metabolic quadratic $D$-space.
Let $L$ be a $D$-lagrangian of $(V,q)$ with a basis $\{v_1,\cdots,v_n\}$ over $D$.
Set $W=(v_2D+\cdots+v_nD)^\perp$.
Since $L=L^\perp\subsetneq W$, one can choose a vector $u_1\in W\setminus L$.
Hence, $u_1\notin (v_1D)^\perp$.
Set $W=u_1D+v_1D$.
By \cref{lemma}, $W$ is a two-dimensional $D$-subspace of $V$, $q|_W$ is $D$-metabolic and $q\simeq_Dq|_W\perp q|_{W^\perp}$.
 We also have $v_2,\cdots,v_n\in W^\perp$.
By dimension count, $v_2D+\cdots+v_nD$ is a $D$-lagrangian of $q|_{W^\perp}$.
The result now follows by induction on $n$.
\end{proof}

\begin{prop}\label{metan}
Let $U$, $V$ and $W$ be finite dimensional right vector spaces over $D$ and let $(V,q)\simeq(U,\rho)\perp (W,\phi)$ be a $D$-isometry of nonsingular quadratic spaces.
Suppose further that $\phi$ is a quadratic $D$-form.
If $q$ and $\phi$ are $D$-metabolic then $\rho$ is also $D$-metabolic.
\end{prop}

\begin{proof}
The proof is very similar to that of \cite[(1.26)]{elman}.
Since $\phi$ is a $D$-metabolic quadratic $D$-form, it is $D$-isometric to orthogonal sums of two-dimensional $D$-metabolic quadratic $D$-forms by \cref{2}.
Hence, it suffices to consider the case where $\dim_DW=2$.
We identify $U$ and $W$ with subspaces of $V$, so that $V=U+W$ and $U\cap W=\{0\}$.
Let $L$ be a $D$-lagrangian of $(V,q)$.
Let $\pi:L\rightarrow W$ be the projection map and set $L_0=\ker\pi=L\cap U$.
If $\pi$ is not surjective, then $\dim_DL_0\geqslant\dim_DL-1$.
Hence, $L_0$ is a $D$-lagrangian of $(U,\rho)$ and the result follows.

Suppose that $\pi$ is surjective, so $\dim_DL_0=\dim_DL-2$.
Choose a $D$-isotropic vector $w\in W$.
As $\pi$ is surjective, there exists $v\in L$ such that $\pi(v)=w$, hence $v=u+w$ for some $u\in U$.
It follows that
\[\rho(ud)=q(ud+wd)-\phi(wd)=q(vd)-\phi(wd)=0\quad {\rm for} \ {\rm all}\ d\in D,\]
 i.e., $u$ is a $D$-isotropic vector of $(U,\rho)$.
Since $L_0\subseteq U$ and $w\in W$, we have $wd\in L_0^{\perp_q}$ for every $d\in D$.
Hence, $ud\in L_0^{\perp_q}$ for every $d\in D$, because $vd=ud+wd\in L\subseteq L_0^{\perp_q}$ for all $d\in D$.
Note that $ud\in U$ and $L_0\subseteq U$, hence $ud\in L_0^{\perp_q}\cap U=L_0^{\perp_\rho}$.

We claim that $w\notin L$.
Since $\phi$ is nonsingular, there exists $w'\in W$ such that
\begin{equation}\label{eq5}
\mathfrak{b}_\phi(w,w')\neq0.
\end{equation}
As $\pi$ is surjective, there exists $v'\in L$ such that $\pi(v')=w'$, i.e., $v'=u'+w'$ for some $u'\in U$.
Now, if $w\in L$ then
\[0=\mathfrak{b}_q(w,v')=\mathfrak{b}_q(w,u'+w')=\mathfrak{b}_q(w,w')=\mathfrak{b}_\phi(w,w'),\]
contradicting (\ref{eq5}).
Hence, $w\notin L$, as claimed.
Since $u+w=v\in L$, we have $u\notin L$, or equivalently, $u\notin L_0$.
By dimension count, $L_0\oplus uD$ is a $D$-lagrangian of $U$, proving the result.
\end{proof}

\begin{lemma}\label{qq}
Let $V$ and $V'$ be right vector spaces over $D$ and let $q:V\rightarrow F$ and $q':V'\rightarrow F$ be nonsingular quadratic forms.
If $q\simeq_Dq'$  then $q\perp (-q')$ is $D$-metabolic.
The converse is also true if $q$ and $q'$ are $D$-anisotropic.
\end{lemma}

\begin{proof}
We identify $V$ and $V'$ with $D$-subspaces of $V\oplus V'$, so that $V\oplus V'=V+V'$ and $V\cap V'=\{0\}$.
Suppose first that $q\simeq_Dq'$.
Let $f:(V,q)\rightarrow (V',q')$ be a $D$-isometry and let $\{v_1,\cdots,v_n\}$ be a basis of $V$ over $D$.
Then the $D$-subspace of $V\oplus V'$ spanned by $v_1+f(v_1),\cdots,v_n+f(v_n)$ is a $D$-lagrangian of $q\perp (-q')$.
Hence, $q\perp (-q')$ is $D$-metabolic.

Conversely, suppose that $q\perp (-q')$ is $D$-metabolic, and the forms $q$ and $q'$ are $D$-anisotropic.
Let $L$ be a $D$-lagrangian of $(V\perp V',q\perp(-q'))$.
Since $q$ is $D$-anisotropic, the intersection $L\cap V$ is trivial.
Hence, the projection $\pi':L\rightarrow V'$ is injective, which implies that $\dim_DL\leqslant \dim_DV'$.
Similarly, we have $\dim_DL\leqslant \dim_DV$.
The equality $\dim_DL=\frac{1}{2}(\dim_DV+\dim_DV')$ yields
\[{\dim}_DV={\dim}_DV'={\dim}_DL.\]
Hence, the projections $\pi:L\rightarrow V$ and $\pi':L\rightarrow V'$ are isomorphisms of right vector spaces over $D$.
It is now readily verified that the map $\pi'\circ\pi^{-1}:V\rightarrow V'$ is a $D$-isometry $(V,q)\simeq_D(V',q')$.
\end{proof}

We are now ready to state a Witt decomposition theorem for quadratic $D$-forms.
\begin{theorem}\label{main}
Let $(V,q)$  be a nonsingular quadratic $D$-space.
Then $q\simeq_Dq_{\met}\perp q_{\an}$, where $q_{\met}$ is $D$-metabolic and $q_{\an}$ is $D$-anisotropic.
Moreover, $q_{\an}$ is uniquely determined, up to $D$-isometry.
\end{theorem}

\begin{proof}
The existence of such a decomposition follows from \cref{dis} and induction on $\dim_DV$.
To prove the uniqueness, suppose that
\[q\simeq_Dq_{\met}\perp q_{\an}\simeq_Dq'_{\met}\perp q'_{\an},\]
where $q_{\an}$ and $q'_{\an}$ are $D$-anisotropic, and $q_{\met}$ and $q'_{\met}$ are $D$-metabolic.
Then
\[q_{\met}\perp q_{\an}\perp(- q'_{\an})\simeq_Dq'_{\met}\perp q'_{\an}\perp(- q'_{\an}),\]
is $D$-metabolic by \cref{qq}.
Note that $q_{\met}$ is a quadratic $D$-form, as it is a subform of $q$.
Since $q_{\met}$ is $D$-metabolic, the form $q_{\an}\perp(- q'_{\an})$ is also $D$-metabolic by \cref{metan}.
Hence, \cref{qq} implies that $q_{\an}\simeq_Dq'_{\an}$.
\end{proof}

Since basic properties of quadratic forms naturally extend to quadratic $D$-forms, one may consider the following question as a generalization of Springer's theorem.

\begin{question}\label{ques}
Let $K/F$ be a finite field extension of odd degree and let $V$ be a right vector space over $F$.
If $q:V\rightarrow F$ is a $D$-anisotropic quadratic $D$-form, does it imply that $q_K:V_K\rightarrow K$ is $D_K$-anisotropic?
\end{question}

\section{The $\pi$-invariant of hermitian and skew hermitian forms}\label{sec-pi}
In this section, we fix $(D,\sigma)$ as a finite dimensional division algebra with involution of the first kind over a field $F$ and $\lambda=\pm1$.
Suppose that either $D\neq F$ or $\lambda\neq-1$, which implies that $\symd_{\lambda}(D,\sigma)\neq\{0\}$, thanks to \cite[(2.6)]{knus}.
We also fix $\pi:\sym_\lambda(D,\sigma)\rightarrow F$ as an $F$-linear map such that $\pi|_{\symd_\lambda(D,\sigma)}$ is nontrivial.

Let $(V,h)$ be a $\lambda$-hermitian space over $(D,\sigma)$.
Define a map
$q_{h,\pi}:V\rightarrow F$
via
\[q_{h,\pi}(v)=\pi(h(v,v))\quad {\rm for}\ v\in V.\]
We call $q_{h,\pi}$ the {\it $\pi$-invariant} of $(V,h)$.

\begin{lemma}\label{quad}
Let $(V,h)$ be a $\lambda$-hermitian space over $(D,\sigma)$.
Then the map $q_{h,\pi}:V\rightarrow F$ is a quadratic $D$-form with the polar form
\begin{align}\label{eq3}
\mathfrak{b}_{h,\pi}(u,v)=\pi(h(u,v)+h(v,u))\quad {\rm for}\ u,v\in V.
\end{align}
\end{lemma}
\begin{proof}
That $q_{h,\pi}$ is a quadratic form and $\mathfrak{b}_{h,\pi}$ is its polar form are easily verified (see \cite[(3.1)]{meh}).
We claim that $W^\perp$ is a vector space over $D$ for every $D$-subspace $W$ of $V$, i.e., $q$ is a quadratic $D$-form.
Let $W$ be a $D$-subspace of $V$ and let $w\in W^\perp$.
We should prove that $wd\in W^\perp$ for every $d\in D$.
Let $d\in D$ and $v\in W$.
If $h(w,v)=0$ then
\[\mathfrak{b}_{h,\pi}(wd,v)=\pi(h(wd,v)+h(v,wd))=0.\]
Otherwise, let $d'=h(w,v)^{-1}h(v,w)d\in D$.
Then $d=h(v,w)^{-1}h(w,v)d'$ and
\begin{align*}
\mathfrak{b}_{h,\pi}(wd,v)&=\pi(h(wd,v)+h(v,wd))\\
&=\pi(\sigma(d)h(w,v)+h(v,w)d)\\
&=\pi(\sigma(d')h(v,w)+h(w,v)d')\\
&=\pi(h(vd',w)+h(w,vd'))\\
&=\mathfrak{b}_{h,\pi}(w,vd')=0.
\end{align*}
Hence, $\mathfrak{b}_{h,\pi}(wd,v)=0$ for all $v\in W$ and $d\in D$.
It follows that $wd\in W^\perp$ for every $d\in D$, proving the claim.
\end{proof}

\begin{lemma}\label{reg}
Let $(V,h)$ be a $\lambda$-hermitian space over $(D,\sigma)$.
Then $h$ is nondegenerate if and only if $q_{h,\pi}$ is nonsingular.
\end{lemma}
\begin{proof}
If there exists $u\in V$ such that $h(u,v)=0$ for all $v\in V$ then $\mathfrak{b}_{h,\pi}(u,v)=0$ for all $v\in V$.
Hence, $h$ is nondegenerate if $q_{h,\pi}$ is nonsingular.
Conversely, suppose that $h$ is nondegenerate.
Choose $x\in\symd_\lambda(D,\sigma)$ such that $\pi(x)\neq0$.
Let $v\in V$ be an arbitrary nonzero vector.
By \cite[(3.5)]{meh} there exists $w\in V$ such that $h(v,w)+h(w,v)=x$.
It follows that
\[\mathfrak{b}_{h,\pi}(v,w)=\pi(h(v,w)+h(w,v))=\pi(x)\neq0,\] i.e., $q_{h,\pi}$ is nonsingular.
\end{proof}

\begin{remark}\label{remp}
Let $(V,h)$ be a $\lambda$-hermitian space over $(D,\sigma)$.
It is worth noting that if $\pi|_{\symd_\lambda(D,\sigma)}$ is trivial then $q_{h,\pi}$ is totally singular.
Indeed, if $\car F\neq2$ then $\pi$ is trivial, because $\sym_\lambda(D,\sigma)=\symd_\lambda(D,\sigma)$.
Hence, $q_{h,\pi}$ is the zero form and does not give any information about $h$.
Otherwise, since $h(u,v)+h(v,u)\in\symd_\lambda(D,\sigma)$ for every $u,v\in V$, assuming $\pi|_{\symd_\lambda(D,\sigma)}=0$, one concludes that
\[\mathfrak{b}_{h,\pi}(u,v)=\pi(h(u,v)+h(v,u))=0\quad {\rm for \ all} \ u,v\in V.\]
Hence, $q_{h,\pi}$ is totally singular.
\end{remark}
Note that \cref{remp} also applies in the exceptional case where $D=F$ and $\lambda=-1$.
In this case, one has $\symd_\lambda(D,\sigma)=\{0\}$.
Hence,  $\pi|_{\symd_\lambda(D,\sigma)}$ is trivial, which implies that the form $q_{h,\pi}$ is totally singular.

\begin{lemma}\label{ker}
Let $(A,\tau)$ be a central simple algebra with involution of the first kind over $F$ and let $S$ be a subspace of $A$.
If there exists a unit $x\in \sym_\lambda(A,\tau)$ such that $\tau(y)xy\in S$ for every $y\in A$, then $\symd_\lambda(A,\tau)\subseteq S$.
\end{lemma}

\begin{proof}
Let $z\in\symd_\lambda(A,\tau)$ be an arbitrary element.
We prove that $z\in S$.
Write $z=z'+\lambda\tau(z')$ for some $z'\in A$ and set $y=x^{-1}z'\in A$.
Then
\begin{align*}
z&=z'+\lambda\tau(z')=xy+\lambda\tau(xy)=xy+\tau(y)x\\
&=\tau(y+1)x(y+1)-\tau(y)xy-\tau(1)\cdot x\cdot1\in S.\qedhere
\end{align*}
\end{proof}

We now show that the form $q_{h,\pi}$ can be used to classify hermitian and skew hermitian forms, up to isometry.

\begin{theorem}\label{isom}
Let $(V,h)$ and $(V',h')$ be two $\lambda$-hermitian spaces over $(D,\sigma)$.
Then $h\simeq h'$ if and only if $q_{h,\pi}\simeq_Dq_{h',\pi}$.
\end{theorem}

\begin{proof}
If $f:(V,h)\simeq(V',h')$ is an isometry then $f$ is an isomorphism of right vector spaces satisfying
\begin{align*}
q_{h',\pi}(f(v))=\pi(h'(f(v),f(v)))=\pi(h(v,v))=q_{h,\pi}(v),
\end{align*}
for every $v\in V$.
Hence, $f:(V,q_{h,\pi})\rightarrow (V',q_{h',\pi})$ is a $D$-isometry.

Conversely, let $f:(V,q_{h,\pi})\simeq_D(V',q_{h',\pi})$ be a $D$-isometry.
We prove that $f:(V,h)\simeq(V',h')$ is an isometry.
Since $f:V\rightarrow V'$ is an isomorphism of right vector spaces over $D$, it suffices to show that
\begin{equation}\label{eq2}
h'(f(u),f(v))=h(u,v)\quad  {\rm for\ every} \  u,v \in V.
\end{equation}
The equality $q_{h',\pi}(f(v))=q_{h,\pi}(v)$ for $v\in V$ implies that
\[h'(f(v),f(v))-h(v,v)\in\ker \pi\quad {\rm for\ every}\ v\in V. \]
If $h'(f(v),f(v))\neq h(v,v)$ for some $v\in V$ then $x:=h'(f(v),f(v))-h(v,v)\in\ker \pi\cap \sym_\lambda(D,\sigma)$ is a nonzero element satisfying
\[\sigma(d)xd=h'(f(vd),f(vd))-h(vd,vd)\in \ker \pi\quad {\rm for\ every} \ d\in D.\]
Hence, $\symd_\lambda(D,\sigma)\subseteq \ker \pi$ by \cref{ker}, contradicting $\pi|_{\symd_\lambda(D,\sigma)}\neq0$.
Thus,
\begin{equation}\label{eq1}
h'(f(v),f(v))= h(v,v)\quad {\rm for\ every} \ v\in V.
\end{equation}
Finally, it is easily seen that the map $h'':V\times V\rightarrow D$ defined by
\[h''(u,v)=h(u,v)-h'(f(u),f(v)),\] is a $\lambda$-hermitian form.
Using (\ref{eq1}), one has $h''(v,v)=0$ for all $v\in V$.
The assumption $D\neq F$ or $\lambda\neq-1$ implies that $h''$ is trivial, hence $h'(f(u),f(v))=h(u,v)$ for all $u,v\in V$, proving (\ref{eq2}).
\end{proof}

\begin{remark}
\cref{isom} does not necessarily hold if $\pi|_{\symd_\lambda(D,\sigma)}$ is trivial.
This is obvious if $\car F\neq 2$, because $q_{h,\pi}$ is trivial in this case.
Suppose that $\car F=2$, so that $\symd(D,\sigma)\subsetneq \sym(D,\sigma)$.
Choose $\alpha\in\sym(D,\sigma)\setminus\symd(D,\sigma)$.
Consider one-dimensional hermitian forms $h$ and $h'$ on $V=D$ satisfying $h(1,1)=\alpha$ and $h'(1,1)=\alpha+\beta$, where $\beta\in\symd(D,\sigma)$ is a nonzero element.
Then for every $d\in D$, one has
\begin{align*}
q_{h',\pi}(d)&=\pi(h'(d,d))=\pi(\sigma(d)(\alpha+\beta)d)\\&=\pi(\sigma(d)\alpha d)=\pi(h(d,d))=q_{h,\pi}(d),
\end{align*}
because $\sigma(d)\beta d\in\symd(D,\sigma)$ and $\pi|_{\symd(D,\sigma)}=0$.
Hence, $q_{h',\pi}\simeq_Dq_{h,\pi}$.
However, $h$ and $h$ are not isometric, since otherwise there exists $d\in D$ such that $h(d,d)=\alpha+\beta$, hence
$\sigma(d)\alpha d=\alpha+\beta$.
It then follows that $\sigma(d)\alpha d+\alpha=\beta\in\symd(D,\sigma)$.
Set $x=\sigma(d+1)\alpha(d+1)\in D$.
Then
\begin{align}\label{eq6}
x=\sigma(d)\alpha d+\alpha+\sigma(d)\alpha+\alpha d=\sigma(d)\alpha d+\alpha+\alpha d+\sigma(\alpha d)\in{\symd}(D,\sigma).
\end{align}
Note that as $\beta\neq0$, we have $d\neq1$.
Hence, (\ref{eq6}) leads to the contradiction $\alpha=\sigma((d+1)^{-1})x(d+1)^{-1}\in\symd(D,\sigma)$.
\end{remark}

\begin{prop}\label{iso}
Let $(V,h)$ be a $\lambda$-hermitian space over $(D,\sigma)$.
A nonzero vector $v\in V$ is an isotropic vector of $h$ if and only if it is a $D$-isotropic vector of $q_{h,\pi}$.
In particular, $h$ is isotropic if and only if $q_{h,\pi}$ is $D$-isotropic.
\end{prop}

\begin{proof}
If $v\in V$ is an isotropic vector of $h$, then $h(vd,vd)=0$ for all $d\in D$.
Hence, $q_{h,\pi}|_{vD}=0$, i.e., $v$ is a $D$-isotropic of $q_{h,\pi}$.

Conversely, suppose that $q_{h,\pi}|_{vD}=0$ for some nonzero vector $v\in V$.
We claim that $h(v,v)=0$.
Choose $x\in\symd_\lambda(D,\sigma)$ such that $\pi(x)\neq0$.
Write $x=y+\lambda\sigma(y)$ for some $y\in D$.
Suppose that $h(v,v)\neq0$ and set $d=h(v,v)^{-1}y\in D$.
Then
\begin{align*}
  \mathfrak{b}_{h,\pi}(v,vd)&=\pi(h(v,vd)+h(vd,v))\\
  &=\pi(h(v,v)d+\lambda\sigma(h(v,v)d))\\
  &=\pi(y+\lambda\sigma(y))=\pi(x)\neq0.
\end{align*}
This contradicts the assumption $q_{h,\pi}|_{vD}=0$.
\end{proof}

\begin{cor}\label{met}
Let $(V,h)$ be a $\lambda$-hermitian space over $(D,\sigma)$.
Then $h$ is metabolic if and only if $q_{h,\pi}$ is $D$-metabolic.
\end{cor}

\begin{proof}
Suppose first that $h$ is metabolic.
Then $h$ is nondegenerate, hence $q_{h,\pi}$ is nonsingular by \cref{quad}.
Let $L$ be a lagrangian of $h$.
Then $L$ is a $D$-subspace of $V$ satisfying $\dim_DL=\frac{1}{2}\dim_DV$ and $h|_{L\times L}=0$.
It follows that $q_{h,\pi}|_L=0$.
Hence, $L$  is a $D$-lagrangian of $q_{h,\pi}$, i.e., $q_{h,\pi}$ is $D$-metabolic.

Conversely, Suppose that $q_{h,\pi}$ is $D$-metabolic.
Then it is nonsingular, hence $h$ is nondegenerate by \cref{quad}.
Let $L$ be a $D$-lagrangian of $q_{h,\pi}$.
Since every $v\in L$ is a $D$-isotropic vector of $q_{h,\pi}$, \cref{iso} shows that $h(v,v)=0$ for every $v\in V$.
The assumption $D\neq F$ or $\lambda\neq-1$ then implies that $h|_{L\times L}=0$.
Hence, $L$ is a lagrangian of $(V,h)$, i.e., $h$ is metabolic.
\end{proof}

\begin{remark}
The `if' implications in \cref{iso} and \cref{met} are not necessarily true if $\pi|_{\symd_\lambda(D,\sigma)}$ is trivial.
This is obvious if $\car F\neq2$, because $q_{h,\pi}$ is trivial in this case.
If $\car F=2$, as already observed in \cref{remp}, the form $q_{h,\pi}$ is totally singular.
Hence, it cannot be metabolic, even if $h$ is metabolic.
This proves the claim for metabolicity.
Finally, to prove the claim for isotropy, let $(V,h)$ be an anisotropic hermitian space over $(D,\sigma)$ satisfying $h(v,v)\in\symd(D,\sigma)$ for some nonzero vector $v\in V$.
Since $\sigma(d)h(v,v)d\in\symd(D,\sigma)$ for every $d\in D$ we have
$q_{h,\pi}(vd)=\pi(\sigma(d)h(v,v)d)=0$ for all $d\in D$.
Hence, $q_{h,\pi}$ is $D$-isotropic.
\end{remark}

The following result is easily verified.
\begin{lemma}\label{pe}
If $(V,h)$ and $(V',h')$ are two $\lambda$-hermitian spaces over $(D,\sigma)$ then $q_{h\perp h',\pi}\simeq_D q_{h,\pi}\perp q_{h',\pi}$.
In particular, $q_{h,\pi}$ and $q_{h',\pi}$ are $D$-compatible.
\end{lemma}

\begin{remark}
Let $(V,h)$ and $(V',h')$ be two $\lambda$-hermitian spaces over $(D,\sigma)$.
In the case where $(V,h)$ and $(V',h')$ are even, one can find a shorter proof of the `if' implication in \cref{isom} as follows (compare \cite[(4.5)]{meh}):
since every $\lambda$-hermitian form is an orthogonal sum of a zero form and a nondegenerate form, it suffices to prove the claim in the case where $h$ and $h'$ are nondegenerate.
If $q_{h,\pi}\simeq_Dq_{h',\pi}$ then using \cref{pe}, one has $q_{h\perp(-h'),\pi}\simeq_D q_{h,\pi}\perp -(q_{h',\pi})$.
Hence, $q_{h\perp(-h'),\pi}$ is $D$-metabolic by \cref{qq}, which implies that $h\perp(-h')$ is metabolic, thanks to \cref{met}.
It follows from \cite[Ch. I, (6.4.5)]{knus1} that $h\simeq h'$.
\end{remark}

Let $K/F$ be a finite field extension such that $D_K:=D\otimes_FK$ is a division algebra and let $(V,h)$ be a $\lambda$-hermitian space over $(D,\sigma)$.
Then there exists a $\lambda$-hermitian space $(V_K,h_K)$ over $(D_K,\sigma_K)$ satisfying
\[h_K(u\otimes a,v\otimes b)=h(u,v)\otimes ab\quad {\rm for \ all} \ u,v\in V\ {\rm and} \ a,b\in K,\]
where $\sigma_K=\sigma\otimes\id$.
Note that $\sym_\lambda(D_K,\sigma_K)=\sym_\lambda(D,\sigma)\otimes K$, hence the map $\pi$ induces a $K$-linear map
\[\pi_K:{\sym}_{\lambda}(D_K,\sigma_K)\rightarrow K,\]
satisfying $\pi_K(x\otimes a)=a\pi(x)$ for all $x\in \sym_\lambda(D,\sigma)$ and $a\in K$.
Therefore, we obtain a quadratic form
$q_{h_K,\pi_K}:V_K\rightarrow K$
satisfying
\[q_{h_K,\pi_K}(v\otimes a)=\pi_K(h(v,v)\otimes a^2)=\pi(h(v,v))a^2=q_{h,\pi}(v)a^2,\]
for all $v\in V$ and $a\in K$.
Clearly, the definition of $q_{h,\pi}$ is functorial, i.e., \[q_{h_K,\pi_K}\simeq(q_{h,\pi})_K.\]

Let $K/F$ be a field extension of odd degree and let $(V,h)$ be an anisotropic $\lambda$-hermitian space over $(D,\sigma)$.
In view of the functoriality of $q_{h,\pi}$, the problem of whether $h_K$ is anisotropic can be generalized to \cref{ques}.

\newpage
\section{Quadratic forms which admit $D$}\label{sec-admit}
We now fix $F$ as a field of characteristic not two and $(D,\sigma)$ as a division algebra with involution of the first kind over $F$.

Let $V$ be a finite dimensional right vector space over $D$ and let $q:V\rightarrow F$ be a quadratic form.
As in \cite[Ch. I, \S 7.4]{knus1}, we say that $q$ {\it admits} $D$ if
\[\mathfrak{b}_q(ud,v)=\mathfrak{b}_q(u,v\sigma(d))\quad {\rm for\ all} \ u,v\in V \ {\rm and} \ d\in D. \]

\begin{lemma}
Every quadratic form which admits $D$ is a quadratic $D$-form.
\end{lemma}

\begin{proof}
Let $V$ be a finite dimensional right vector space over $D$ and let $q:V\rightarrow F$ be a quadratic form which admits $D$.
Let $W$ be a $D$-subspace of $V$ and let $w\in W^\perp$.
Then for every $d\in D$ and $v\in W$ we have
\[\mathfrak{b}_q(wd,v)=\mathfrak{b}_q(w,v\sigma(d))=0.\]
Hence, $wd\in W^\perp$, proving that $W^\perp$ is a $D$-subspace of $V$.
\end{proof}
\begin{exm}
Suppose that $D$ is a quaternion algebra and $\sigma$ is its canonical involution, defined by $\sigma(x)=\trd_D(x)-x$ for $x\in D$, where $\trd_D(x)$ is the reduced trace of $x$ in $D$.
Let $V$ be a finite dimensional right vector space over $D$ and let $q:V\rightarrow F$ be a quadratic form which admits $D$.
Then $q$ is isotropic if and only if it is $D$-isotropic.
This follows from the fact that
\[q(vd)=\frac{1}{2}\mathfrak{b}_q(vd,vd)=\frac{1}{2}\mathfrak{b}_q(v,vd\sigma(d))=\frac{1}{2}\mathfrak{b}_q(v,v)d\sigma(d)=0,\]
for all $v\in V$ and $d\in D$ (note that $d\sigma(d)\in F$ for every $d\in D$).
In particular, using Springer's theorem \cite[(18.5)]{elman}, one can find an affirmative answer to \cref{ques} in this special case.
\end{exm}
\begin{definition}
Let $\pi:\symd_\lambda(D,\sigma)\rightarrow F$ be an $F$-linear map.
We say that $\pi$ is {\it symmetric} if
\[\pi(xy+\lambda\sigma(xy))=\pi(yx+\lambda\sigma(yx))\quad {\rm for \ all} \ x,y\in D.\]
\end{definition}

\begin{lemma}\label{xy}
Let $A$ be a central simple algebra over $F$ and let $l:A\rightarrow F$ be an $F$-linear map.
If $l(xy)=l(yx)$ for all $x,y\in A$ then $l$ is a scalar multiple of the reduced trace $\trd_A:A\rightarrow F$.
\end{lemma}

\begin{proof}
By scalar extension to a splitting field, it is enough to consider the case where $A=M_n(F)$ is the matrix algebra.
Let $e_{ij}\in M_n(F)$ be the matrix whose $ij$-entry is $1$ and all other entries are zero.
Then for all $i,j,k$ we have $l(e_{ij})=l(e_{ik}e_{kj})=l(e_{kj}e_{ik})$.
Hence, $l(e_{ij})=0$ if $i\neq j$.
Also, taking $i=j$, one concludes that $l(e_{ii})=l(e_{kk})$ for all $i,k$.
Hence, $l=\alpha\cdot \trd_A$, where $\alpha=l(e_{11})\in F$.
\end{proof}

\begin{lemma}\label{ex}
Let $\pi:\symd_\lambda(D,\sigma)\rightarrow F$ be an $F$-linear map.
If $\pi$ is symmetric then there exists $\alpha\in F$ such that $\pi=\alpha\cdot\trd_D|_{\symd_\lambda(D,\sigma)}$.
\end{lemma}

\begin{proof}
Define $l_\pi:D\rightarrow F$ via $l_\pi(x)=\frac{1}{2}\pi(x+\lambda \sigma(x))$.
Then $l_\pi$ is an $F$-linear map satisfying
\[l_\pi(xy)=\frac{1}{2}\pi(xy+\lambda\sigma(xy))=\frac{1}{2}\pi(yx+\lambda\sigma(yx))=l_\pi(yx),\]
for all $x,y\in D$.
By \cref{xy}, there exists $\alpha\in F$ such that $l_\pi=\alpha\cdot\trd_D$.
Note that if $x\in\symd_\lambda(D,\sigma)$ then
$l_\pi(x)=\frac{1}{2}\pi(x+\lambda\sigma(x))=\pi(x)$, hence $l_\pi|_{\symd_\lambda(D,\sigma)}=\pi$, proving the claim.
\end{proof}

Let $\pi:\symd_\lambda(D,\sigma)\rightarrow F$ be a symmetric $F$-linear map.
By \cref{ex}, $\pi$  is the restriction to $\symd_\lambda(D,\sigma)$ of a scalar multiple of $\trd_D$.
Since $\trd_D(\sigma(x))=\trd_D(x)$ for all $x\in D$, one concludes that $\pi$ is trivial if $\lambda=-1$.
Hence, for the rest of this section we only consider the case where $\lambda=1$ for studying symmetric linear maps.
\begin{lemma}\label{adm}
Let $(V,h)$ be a nontrivial hermitian space over $(D,\sigma)$ and let $\pi:\symd(D,\sigma)\rightarrow F$ be a nonzero linear map.
Then $q_{h,\pi}$ admits $D$ if and only if $\pi$ is symmetric.
\end{lemma}

\begin{proof}
Suppose first that $q_{h,\pi}$ admits $D$.
Let $x,y\in D$ and set $d=\sigma(x)\in D$.
As $h$ is nontrivial there exist $u,v\in V$ such that $h(u,v)=y$.
Then
\begin{align*}
\pi(xy+\sigma(xy))&=\pi(\sigma(d)h(u,v)+\sigma(\sigma(d)h(u,v)))\\
&=\pi(h(ud,v)+\sigma(h(ud,v)))\\&=\pi(h(ud,v)+h(v,ud))\\
&=\mathfrak{b}_{h,\pi}(ud,v)=\mathfrak{b}_{h,\pi}(u,v\sigma(d))\\&=\pi(h(u,v\sigma(d))+h(v\sigma(d),u))\\
&=\pi(h(u,v)\sigma(d)+\sigma(h(u,v)\sigma(d)))\\&=\pi(yx+\sigma(yx)).
\end{align*}
Hence, $\pi$ is symmetric.
Conversely, suppose that $\pi$ is symmetric.
By \cref{ex}, $\pi=\alpha\cdot\trd_D$ for some $\alpha\in F$.
Hence, for every $u,v\in V$ and $d\in D$ we have
\begin{align*}
\mathfrak{b}_{h,\pi}(ud,v)&=\pi(h(ud,v)+h(v,ud))
\\&=\alpha{\trd}_D(\sigma(d)h(u,v)+h(v,u)d)\\
&=\alpha{\trd}_D(h(u,v)\sigma(d)+dh(v,u))\\
&=\alpha{\trd}_D(h(u,v\sigma(d))+h(v\sigma(d),u))\\
&=\pi(h(u,v\sigma(d))+h(v\sigma(d),u))=\mathfrak{b}_{h,\pi}(u,v\sigma(d)).\qedhere
\end{align*}
Hence, $q_{h,\pi}$ admits $D$.
\end{proof}

Let $(V,q)$ be a quadratic $D$-space.
Since the $\pi$-invariant of every hermitian form on $V$ is a quadratic $D$-form, a natural question is whether every quadratic $D$-form can be realised as the $\pi$-invariant of a hermitian form.
Using \cite[(2.3)]{lewis}, one can find a solution to this question in the case where $\pi$ is symmetric:

\begin{prop}\label{real}
Let $(V,q)$ be a quadratic $D$-space and let $\pi:\symd(D,\sigma)\rightarrow F$ be a nonzero symmetric $F$-linear map.
Then there exists a hermitian form $h$ on $V$ such that $q=q_{h,\pi}$ if and only if $q$ admits $D$.
\end{prop}

\begin{proof}
The `only if' implication follows from \cref{adm}.
Conversely, suppose that $\pi$ is symmetric.
Then $\pi=\alpha\cdot\trd_D$ for some $\alpha\in F$, thanks to \cref{ex}.
Since $\trd_D(xy)=\trd_D(yx)$ for all $x,y\in D$, the assignment $(x,y)\mapsto \alpha\trd_D(xy)$ induces a symmetric bilinear form $\mathfrak{b}:D\times D\rightarrow F$.
Note that $\mathfrak{b}$ is associative, i.e.,
\[\mathfrak{b}(x,yz)=\mathfrak{b}(xy,z)\quad {\rm for \ all} \ x,y,z\in D.\]
Hence, $D$ is a symmetric algebra, in the sense of \cite{lewis}.
By  \cite[(2.3)]{lewis} there exists a hermitian form $h:V\times V\rightarrow D$ such that
$\mathfrak{b}(h(x,y),1)=\mathfrak{b}_q(x,y)$ for all $x,y\in V$.
Hence, 
\[\alpha{\trd}_D(h(x,x))=\mathfrak{b}(h(x,x),1)=\mathfrak{b}_q(x,x)=2q(x)\quad {\rm for \ all\ } x\in V.\]
It follows that 
$\frac{1}{2}\pi(h(x,x))=q(x)$ for every $x\in V$.
 The scaled form $\frac{1}{2}\cdot h$ is therefore the required hermitian form.
\end{proof}

\noindent{\sc A.-H. Nokhodkar, {\tt
    a.nokhodkar@kashanu.ac.ir},\\
Department of Pure Mathematics, Faculty of Science, University of Kashan, P.~O. Box 87317-51167, Kashan, Iran.}

\end{document}